\documentclass[reqno]{amsproc}
\usepackage[utf8]{inputenc}
\usepackage{amsfonts}
\usepackage{graphicx}
\usepackage{amscd}
\usepackage{amsmath}
\usepackage{amssymb}
\usepackage{latexsym}
\usepackage{hyperref}
\usepackage[all]{xy}
\usepackage{color}
\usepackage{mathrsfs}
\theoremstyle{plain}

\newtheorem{corollary}{\bf Corollary}

\newtheorem{example}{\bf Example}
\newtheorem{lemma}{\bf Lemma}

\newtheorem{proposition}{\bf Proposition}
\newtheorem{remark}{\bf Remark}

\newtheorem{theorem}{\bf Theorem}
\numberwithin{equation}{section}

\usepackage{xcolor}

\begin{document}
\title[Gradient Einstein solitons]{Geometric and analytic results for Einstein solitons}
\author[Enrique F. L. Agila]{Enrique Fernando López Agila$^1$}
\author[José N. V. Gomes]{José Nazareno Vieira Gomes$^2$}
\address{$^1$Facultad de Ciencias Naturales y Matemáticas, Escuela Superior Politécnica del Litoral, Guayaquil, Ecuador.}
\address{$^{2}$Departamento de Matemática, Centro de Ciências Exatas e Tecnologia, Universidade Federal de São Carlos, São Paulo, Brazil.}
\email{$^1$enfelope@espol.edu.ec}
\email{$^2$jnvgomes@ufscar.br}
\urladdr{$^1$https://www.fcnm.espol.edu.ec}
\urladdr{$^2$https://www2.ufscar.br}
\begin{abstract}
We compute a lower bound for the scalar curvature of a gradient Einstein soliton under a certain assumption on its potential function. We establish an asymptotic behavior of the potential function on a noncompact gradient shrinking Einstein soliton. As a result, we obtain the finiteness of its fundamental group and its weighted volume. We also prove some geometric and analytic results for constructing gradient Einstein solitons that are realized as warped metrics, and we give a few explicit examples.
\end{abstract}

\maketitle
\section{Introduction}
Our objective here is to prove some geometric and analytic results for gradient $\rho$-Einstein solitons that have not yet been addressed in the literature. A Riemannian metric $g$ on an $n$-dimensional smooth manifold $M^n$, with $n\geq3$, is a gradient $\rho$-Einstein soliton if there exists a smooth function $\varphi:M^n\rightarrow\mathbb{R}$ such that the Ricci tensor of $g$ is given by
\begin{equation}\label{eqres}
Ric+\nabla^2 \varphi=(\lambda+\rho S)g,
\end{equation}
for some real numbers $\rho\neq0$ and $\lambda$. Here, $\nabla^2 \varphi$ is the Hessian of $\varphi$ and $S=tr(Ric)$ is the scalar curvature of $g$. The function $\varphi$ is called the potential function of a gradient $\rho$-Einstein soliton $(M^n,g,\varphi)$. The soliton is trivial whenever $\nabla\varphi$ is parallel. The reason for considering $n\geq3$ is that, for $n=2$, equation~\eqref{eqres} reduces to the gradient Yamabe solitons equation, which are special solutions to the Yamabe flow. In this case, we refer the reader to Daskalopoulos and Sesum~\cite{daskalo}.

It is known that a gradient $\rho$-Einstein soliton arises as a self-similar solution of the Ricci Bourguignon flow~\cite{bourguignon1981ricci}, namely
\begin{equation*}
\frac{\partial}{\partial t}g(t)=-2 \big(Ric - \rho Sg(t)\big),
\end{equation*} 
with $g(0)=g$ on a smooth manifold $M^n$, which has a unique solution for a positive time interval, when $\rho<\frac{1}{2(n-1)}$ and $M^n$ is compact, for any initial Riemannian metric $g$, see Catino et al.~\cite[Theorem~2.1]{catino2017ricci}. A gradient $\rho$-Einstein soliton is steady for $\lambda=0$, shrinking for $\lambda>0$ and expanding for $\lambda<0$. By a complete gradient $\rho$-Einstein soliton $(M^n,g,\varphi)$, we mean a geodesically complete Riemannian manifold $(M^n,g)$ with a potential function $\varphi$ such that $\nabla\varphi$ is complete.

It is also known that any compact gradient $\rho$-Einstein soliton with $\rho\leq\frac{1}{2(n-1)}$, is either shrinking of positive scalar curvature or it is trivial. Moreover, any compact gradient Schouten soliton (i.e., $\rho=\frac{1}{2(n-1)}$) or traceless Ricci soliton (i.e., $\rho=\frac{1}{n}$) is trivial, see Catino and Mazzieri~\cite[Theorem~3.1]{catino2016gradient}. It is also trivial any complete gradient steady Schouten soliton, see~\cite[Theorem~1.5]{catino2016gradient}. The examples we present in Section~\ref{Sec-Examples} include two nontrivial 3-dimensional complete noncompact traceless Ricci solitons. The fact that they happen to be of dimension three comes from an analytic argument, see Examples~\ref{traceless1} and \ref{traceless2}. Note that the latter are examples of a steady and a shrinking soliton, respectively.

Catino and Mazzieri also proved that any complete gradient Schouten soliton which is either shrinking or steady has nonnegative scalar curvature. This was a consequence of a special case of a more general result on complete ancient solutions for the Schouten flow, see \cite[Proposition~5.1 and Corollary~5.2]{catino2016gradient}. Very recently, Borges~\cite[Theorem~1.1]{valterborges2022} showed that the scalar curvature $S$ of any geodesically complete noncompact gradient Schouten soliton that is either shrinking or expanding is bounded. More precisely, it satisfies $0\leq \lambda S \leq 2(n-1) \lambda^{2}$.

Example~\ref{exemimcp} is a nontrivial incomplete noncompact gradient Schouten soliton with unbounded negative scalar curvature, which can be either steady, shrinking or expanding, and then it shows that the assumption of completeness is essential for the results by Catino and Mazzieri, and by Borges.

In our first theorem, we give our contribution to this subject by considering a special class of gradient $\rho$-Einstein solitons. 

\begin{theorem}\label{ESRBF}
Let $(M^n,g,\varphi)$ be a geodesically complete gradient $\rho$-Einstein soliton with $0<\rho<\frac{1}{2(n-1)}$. Suppose there is a point $p\in M^n$ that satisfies either one of the following two conditions for the scalar curvature and potential function
\begin{enumerate}
\item $S(x)\geq -Ar^2(x)-B\quad\text{and}\quad|\nabla\varphi(x)|\leq Ar(x)+B,\,\text{or}$
 \item $S(x)\geq -A\quad\text{and}\quad \varphi(x)\geq-Ar(x)^2-B,$
\end{enumerate}
for all $r(x)\geq r_0$ and some nonnegative constants $A,B$ and $C$, where $r(x)$ is the distance function from $p$ and $r_0$ is a sufficiently large positive integer. Then,
 \begin{equation*}
 S(x)\geq \min\left\{0,\frac{\lambda n}{1-\rho n}\right\}.
 \end{equation*}
\end{theorem}

\begin{remark}
The potential function assumptions of Theorem~\ref{ESRBF} occur for any geodesically complete noncompact gradient shrinking Schouten soliton, see Borges~\cite{valterborges2022}. For the case of gradient Ricci solitons, see Cao and Zhou~\cite{cao2010complete}, and also  Munteanu and Sesum~\cite{munteanu2013gradien}.
\end{remark}

The proof of Theorem~\ref{ESRBF} follows from a suitable control for the weighed volume growth of geodesic balls together with the weak maximum principle at infinity for an appropriate drifted Laplacian, see Lemmas~\ref{lemvolex1} and~\ref{princwf} for details.

Having proved Theorem~\ref{ESRBF}, we work to establish an asymptotic behavior of the potential function on a geodesically complete noncompact gradient shrinking $\rho$-Einstein soliton. For this, we follow the approach by Cao and Zhou from which we know that the potential function $\varphi$ of an arbitrary complete noncompact gradient shrinking Ricci soliton $(M^n,g)$ has an asymptotic behavior similar to the potential function of a Gaussian shrinking soliton. Indeed, from the proof of Proposition~2.1 in \cite{cao2010complete}, we can conclude that for every minimizing unit speed geodesic $\gamma:[0,\infty)\to M^n$ emanating from a point $p\in M^n$ the potential function along $\gamma$ satisfies
\begin{equation*}
\lim_{t\rightarrow+\infty}\frac{1}{t}\frac{d\varphi}{dt}(t)=\lambda.
\end{equation*}

In our second theorem, we obtain a similar result for the potential function of a geodesically complete noncompact gradient shrinking $\rho$-Einstein soliton with nonnegative scalar curvature and $\rho>0$. Notice that, in the shrinking case, the assumptions in Theorem~\ref{ESRBF} imply that the scalar curvature is nonnegative. Also, Borges' result above mentioned guarantees that any geodesically complete noncompact gradient shrinking Schouten soliton has nonnegative scalar curvature. 

\begin{theorem}\label{teocresc2}
Let $(M^n,g,\varphi)$ be a geodesically complete noncompact gradient shrinking $\rho$-Einstein soliton with nonnegative scalar curvature and $\rho>0$. Then:
\begin{enumerate}
\item For any fixed point $p\in M^n$, and every minimizing unit speed geodesic $\gamma:[0,\infty)\rightarrow M^n$ emanating from $p$, we obtain
\begin{equation}\label{crepote}
\limsup_{t\rightarrow+\infty}\frac{1}{t}\frac{d\varphi}{dt}\geq\lambda
\end{equation}
uniformly along $\gamma$, i.e., it does not depend on the initial direction of the geodesic.
\item $M^n$ has finite fundamental group, and the  weighted volume $\operatorname{vol}_\varphi(M)$ is finite. In particular, the weak maximum principle at infinity for  $\Delta_\varphi$ holds on $(M^n,g)$.
\end{enumerate}
\end{theorem}

As an application of Theorem~\ref{teocresc2}, we obtain a compactness criteria for a class of gradient shrinking $\rho$-Einstein solitons as follows.

\begin{corollary}\label{Comp-RES}
Let $(M^n,g,\varphi)$ be a geodesically complete gradient shrinking $\rho$-Einstein soliton with nonnegative scalar curvature and $\rho>0$. Suppose there exists a point $p\in M^n$ such that for all $r(x)\geq r_0$, $|\nabla \varphi(x)|\leq \left(\lambda-\varepsilon\right)r(x)+c$, for some positive constants $c$ and $\varepsilon$, with $\varepsilon\leq \lambda$, where $r(x)$ is the distance function from $p$ and $r_0$ is a sufficiently large positive integer. Then, the manifold $M^n$ is compact.
\end{corollary}

The proof of Theorem~\ref{teocresc2} follows from a more general setting that we discuss in Section~\ref{Abpf}, which is motivated by the construction of gradient $\rho$-Einstein solitons that are realized as warped metrics $g=g_B+f^2g_F$ on the product manifold $B^n\times F^m$, see Section~\ref{SecGESWP}. For simplicity, we say that $B^n\times_f F^m$ is a \emph{gradient $\rho$-Einstein soliton warped product} with a potential function $\eta$ and a warping function $f$. The notation $\tilde\psi$ stands for the lift of a smooth function $\psi$ on $B^n$ to  $B^n\times F^m$. The main result of this subject is:

\begin{theorem}\label{rhoonlybase}
Let $B^{n}\times_{f} F^{m}$ be a geodesically complete gradient $\rho$-Einstein soliton warped product with a nonconstant warping function $f$. Suppose that either one
of the following three conditions holds
\begin{enumerate}
\item \label{rhoonlybase-1} $m\neq2,$ or
\item \label{rhoonlybase-1-1} $\rho\neq 1/6,$ or
\item \label{rhoonlybase-2} $f$ is bounded.
\end{enumerate}
Then, the scalar curvature of the fiber $(F^m,g_{F})$ is constant. Moreover, the potential function $\eta=\tilde{\varphi}$ and $\lambda+\rho S_g=\tilde{\Lambda}$, for some smooth functions $\varphi$ and $\Lambda$ on $B^n.$
\end{theorem}

We observe that Theorem~\ref{rhoonlybase} is simpler under the assumption that the base is compact, see Proposition~\ref{compct-casePF}. 

The next question then is, what are the necessary and sufficient conditions for constructing a gradient $\rho$-Einstein soliton warped product? This is the content of Propositions~\ref{condsufrho} and \ref{constrhoe}.

We highlight that Examples~\ref{traceless1}, \ref{traceless2} and~\ref{exemimcp} above mentioned are of gradient $\rho$-Einstein soliton warped products. We establish the sign of the scalar curvature of the fiber of a gradient $\rho$-Einstein soliton warped product which is either shrinking or steady, provided the warping function reaches a minimum, see Corollary~\ref{teofitrwp}. We also establish a compactness criteria for a geodesically complete gradient shrinking $\rho$-Einstein soliton warped product by means of an asymptotic behavior of the potential function, see Corollary~\ref{cricomphowp}.

\section{Lower bound for the scalar curvature}
In this section, we obtain a lower bound for the scalar curvature of a geodesically complete gradient $\rho$-Einstein soliton via a suitable control for the volume growth of geodesic balls and the weak maximum principle at infinity for an appropriate drifted Laplacian.

Let $(M^n,g)$ be a Riemannian manifold with a weighted volume form $e^{-h}d\operatorname{vol}$, for some $h\in C^\infty(M)$. In this context, we use the Bakry-Emery Ricci tensor
\begin{equation*}
Ric_h=Ric+\nabla^2 h
\end{equation*}
and the drifted Laplacian
\begin{equation*}
\Delta_hu=e^h \operatorname{div}\left(e^{-h}\nabla u\right).
\end{equation*}

Let $B(p,R)\subset M^n$ be the metric ball of radius $R>0$ and centered at $p\in M^n$ with boundary $\partial B(p,R)$. We set
\begin{equation*}
\operatorname{vol}_h(B(p,R))=\int_{B(p,R)}\!\!\!\!e^{-h}d\operatorname{vol}\qquad and\qquad \operatorname{vol}_h(\partial B(p,R))=\int_{\partial B(p,R)}\!\!\!\!\!\!e^{-h}d\operatorname{vol}_{n-1},
\end{equation*}
where $e^{-h}d\operatorname{vol}_{n-1}$ stands for the weighted volume form on $\partial B(p,R)$.

In order to obtain a suitable control for the volume growth of geodesic balls, we make use of two classical tools in Riemannian geometry. The first one is the Bochner technique, which we apply to the gradient of the distance function to derive a Riccati inequality. The second one is the first variation of the area formula, which allows us to estimate the weighted volume of geodesic balls in polar coordinates.
\begin{lemma}\label{lemvolex1}
Let $(M^n,g,\varphi)$ be a geodesically complete gradient $\rho$-Einstein soliton with $0<\rho<\frac{1}{2(n-1)}$. Suppose that the hypotheses of Theorem~\ref{ESRBF} are satisfied. Then, for a fixed $r_0>0$ and for all $R\geq r_0,$ there exist constants $A_0$, $B_0$ and $C_0$ that do not depend on $R$, such that
\begin{equation}\label{inelel}
\operatorname{vol}_{\hat\varphi}(B(p,R))\leq  A_0+B_0\int^{R}_{r_0}e^{C_0 t^2}dt,
\end{equation}
where $\hat\varphi=\frac{1}{1-2\rho(n-1)}\varphi$.
\end{lemma}
\begin{proof}
Let $B(p,R)\subset M^n$ be a geodesic ball of radius $R>0$ and centered at $p\in M^n$. Consider the distance function $r(x)$ from $p$, and apply the classical Bochner formula for $r(x)$ to get
\begin{equation*}
0=\left|\nabla^{2} r\right|^{2}+R i c(\partial r, \partial r)+g\left(\nabla \Delta r, \partial r\right) .
\end{equation*}
By Schwarz inequality
\begin{equation*}
0 \geq \frac{(\Delta r)^{2}}{n-1}+R i c(\partial r, \partial r)+g\left(\nabla \Delta r, \partial r\right).
\end{equation*}
Define $\hat\varphi=\frac{1}{1-2\rho(n-1)}\varphi$, $z=\Delta r$ and $z_{\hat\varphi}=z-\langle \nabla\hat\varphi,\partial r\rangle$. Then $\dot z=g(\nabla\Delta r,\partial r)$ and  $\dot z_{\hat\varphi}=\dot z-\nabla^2\hat\varphi(\partial r,\partial r)$. Thus, the previous inequality becomes
\begin{eqnarray*}
\dot z_{\hat\varphi}\leq -\frac{z^2}{n-1}-Ric_{\hat\varphi}(\partial r,\partial r),
\end{eqnarray*}
which is known as a Riccati inequality. Use the gradient $\rho$-Einstein soliton equation~\eqref{eqres} to obtain
\begin{eqnarray*}
\dot z_{\hat\varphi}+\frac{z^2}{n-1}\leq-(\lambda+\rho S)-2\rho(n-1)\nabla^2\hat\varphi(\partial r,\partial r).
\end{eqnarray*}
Now, suppose that the hypotheses of the first item of Theorem~\ref{ESRBF} are satisfied, then
\begin{eqnarray*}
\dot z_{\hat\varphi}+\frac{z^2}{n-1}\leq-\lambda+\rho A r^2+B\rho-2\rho(n-1)\nabla^2\hat\varphi(\partial r,\partial r).
\end{eqnarray*}
Integrating the previous inequality from a fixed $r_0>0$ to $r>r_0$ and using the assumption that $|\nabla\varphi|\leq Ar+B$ gives
\begin{eqnarray*}
 z_{\hat\varphi}+\frac{1}{n-1}\int_{r_0}^r z^2(t)dt\leq \rho A r^3+C_1,
\end{eqnarray*}
for some positive constant $C_1$ (which depends only on $r_0$) or, equivalently,
\begin{equation*}
 z_{\hat\varphi}+\frac{1}{n-1}\int_{r_0}^r (z_{\hat\varphi}+\hat\varphi'(t))^2dt\leq \rho A r^3+C_1.
\end{equation*}
On the other hand, Jensen's inequality implies
\begin{equation*}
    \int_{r_0}^r (z_{\hat\varphi}+\hat\varphi'(t))^2dt\geq \frac{1}{r-r_0}\left(\int_{r_0}^r(z_{\hat\varphi}+\hat\varphi'(t))dt\right)^2.
\end{equation*}
Combining the latter inequalities, we have
\begin{equation}\label{chav}
 z_{\hat\varphi}+\frac{1}{r(n-1)}\left(\int_{r_0}^r(z_{\hat\varphi}+\hat\varphi'(t))dt\right)^2\leq \rho A r^3+C_1.
\end{equation}
Besides, again from assumption $|\nabla\varphi|\leq Ar+B$ and the fundamental theorem of calculus, we get
\begin{equation}\label{potenes}
    |\varphi(x)-\varphi(p)|\leq C_2+Ar^2(x),
\end{equation}
for some constant $C_2>0$. Now we claim that 
\begin{equation}\label{desilpe}
\int_{r_0}^r z_{\hat\varphi}(t)dt\leq \beta(C_2+2C_3)+\sqrt{2C_1(n-1)r}+\beta A r^2+\sqrt{(n-1)\rho A}r^{2},
\end{equation}
for any $r\geq r_0$, where $C_3=\max\{0,\max_{x\in B(p,r_0)}\varphi(x)\}$ and $\beta=\frac{1}{1-2\rho(n-1)}$. Indeed, we define
\begin{equation}\label{wmes}
    w(r):=\beta(C_2+2C_3)+\sqrt{2C_1(n-1)r}+\beta A r^2+\sqrt{(n-1)\rho A}r^{2}-\int_{r_0}^r z_{\hat\varphi}(t)dt,
\end{equation}
and then $w(r_0)>0$. We assume that $w$ is not always positive for $r\geq r_0$ so that we can take the first number $\hat{r}>r_0$ such that $w(\hat{r})=0$. Thus,
\begin{equation*}
    \beta(C_2+2C_3)+\sqrt{2C_1(n-1)\hat{r}}+\beta A \hat{r}^2+\sqrt{(n-1)\rho A}\hat{r}^{2}=\int_{r_0}^{\hat{r}} z_{\hat\varphi}(t)dt.
\end{equation*}
By~\eqref{potenes} we get $\hat\varphi(\hat{r})-\hat\varphi(r_0)+\beta(C_2+2C_3)+\beta A \hat{r}^2\geq 0$, and then
\begin{align*}
&\frac{1}{\hat{r}(n-1)}\left(\hat\varphi(\hat{r})-\hat\varphi(r_0)+\int_{r_0}^{\hat{r}} z_{\hat\varphi}(t)dt\right)^2\\
&=\frac{1}{\hat{r}(n-1)}\!\left(\hat\varphi(\hat{r})\!-\!\hat\varphi(r_0)\!+\!\beta(C_2+2C_3)\!+\!\sqrt{2C_1(n-1)\hat{r}}\!+\!\beta A \hat{r}^2\!+\!\sqrt{(n-1)\rho A}\hat{r}^{2}\right)^2\\
&\geq \frac{1}{\hat{r}(n-1)}\left(\sqrt{2C_1(n-1)\hat{r}}+\sqrt{(n-1)\rho A}\hat{r}^{2}\right)^2\geq 2C_1+\rho A \hat{r}^3.
\end{align*}
Combining the previous inequality with~\eqref{chav} one has $z_{\tilde\varphi}(\hat{r})\leq -C_1<0$. Taking the derivative in~\eqref{wmes}, we get $w'(\hat{r})>0$. Consequently, there exists $\varepsilon>0$ such that $w(r)< w(\hat{r})=0$ for all $r\in (\hat{r}-\varepsilon,\hat{r})$, this contradicts the choice of $\hat{r}$ and proves our claim.

Now, consider the annular set $B(p,r_0,R)=B(p,R)\setminus B(p,r_0)$, and note that
\begin{equation*}
\operatorname{vol}_{\hat\varphi}(B(p,R))=\operatorname{vol}_{\hat\varphi}(B(p,r_0))+\operatorname{vol}_{\hat\varphi}(B(p,r_0,R)).
\end{equation*}

Next, define $A_0=\operatorname{vol}_{\hat\varphi}(B(p,r_0))$ so that it is enough to control $\operatorname{vol}_{\hat\varphi}(B(p,r_0,R))$. In order to do so, we can use exponential coordinates around $p$ and we can write $d\operatorname{vol}=\mathcal{A}(r,\theta)dr\wedge d\theta_{n-1}$, where $d\theta_{n-1}$ is the standard volume element on the unit sphere $S^{n-1}$. Define $\mathcal{A}_{\hat\varphi}(r,\theta)=e^{-\hat\varphi}\mathcal{A}(r,\theta)$, and observe that by the first variation of the area formula, one has
\begin{equation*}
\frac{\mathcal{A}_{\hat\varphi}(r,\theta)}{\mathcal{A}_{\hat\varphi}(r_0,\theta)}=e^{\int_{r_0}^{r}z_{\hat\varphi}(t,\theta)dt}
\end{equation*}
for $r\geq r_0>0$, see, e.g., Zhu~\cite{zhu1997comparison} or Wei and Wylie~\cite{wei2009comparison}. So, \eqref{desilpe} implies that
\begin{eqnarray*}
\mathcal{A}_{\hat\varphi}(r,\theta)\leq  \mathcal{A}_{\hat\varphi}(r_0,\theta)e^{ C_0 r^2+B_1},
\end{eqnarray*}
 for some positive constants $C_0$ and $B_1$ (that depend only on $r_0$). Whence
\begin{eqnarray*}
\operatorname{vol}_{\hat\varphi}(B(p,r_0,R))&=&\int_{r_0}^R\int_{S^{n-1}}\mathcal{A}_{\hat\varphi}(r,\theta)d\theta dr\\
&\leq& \int_{r_0}^R\int_{S^{n-1}}\mathcal{A}_{\hat\varphi}(r_0,\theta)e^{ C_0 r^2+B_1}d\theta dr\\
&=& A_{\hat\varphi}(r_0)\int_{r_0}^R e^{ C_0r^2+B_1} dr,
\end{eqnarray*}
where
\begin{equation*}
 A_{\hat\varphi}(r_0)=\int_{S^{n-1}}\mathcal{A}_{\hat\varphi}(r_0,\theta)d\theta.
\end{equation*}
This is enough to conclude the proof of the lemma under the hypotheses of the first item of Theorem~\ref{ESRBF}. Under the hypotheses of the second item of Theorem~\ref{ESRBF}, the proof is simpler and analogous to the first part.
\end{proof}

We now guarantee the validity of the weak maximum principle at infinity for an appropriate drifted Laplacian on a class of gradient $\rho$-Einstein solitons. Recall that the weak maximum principle at infinity for the drifted Laplacian $\Delta_{h}$ holds on $(M^n,g)$ if given a $C^2$ function $u:M^n\rightarrow\mathbb{R}$ with $\sup_M u=u^*<+\infty,$ there exists a sequence $\{x_k\}$ in $M^n$ such that, for all $k\in \mathbb{N},$
\begin{equation*}
(i)\; u(x_k) > u^*-\frac{1}{k}\quad\text{and}\quad (ii)\;\Delta_{h}u(x_k)<\frac{1}{k}.
\end{equation*}

We know from Pigola, Rigoli and Setti~\cite{pigola2005maximum} that the weak maximum principle at infinity is valid for the drifted Laplacian $\Delta_\varphi$ on any geodesically complete gradient Ricci soliton with potential function $\varphi$. Indeed, as has already been observed by Pigola, Rimoldi and Setti~\cite{pigola2011remarks}, one can use the weighted volume growth by Wei and Wylie~\cite[Theorem~4]{wei2009comparison} to get such a principle from \cite[Chapter~3]{pigola2005maximum} or, alternatively, see \cite[Corollary~10]{pigola2011remarks}. Our case follows as in the previous two papers by using Lemma~\ref{lemvolex1}. For clarity, we give the proof here.

\begin{lemma}\label{princwf}
Let $(M^n,g,\varphi)$ be a geodesically complete gradient $\rho$-Einstein soliton with $0<\rho<\frac{1}{2(n-1)}$. Suppose that the hypotheses of Theorem~\ref{ESRBF} are satisfied. Then, the weak maximum principle at infinity for the drifted Laplacian $\Delta_{\hat\varphi}$ holds on $(M^n,g)$, where $\hat\varphi=\frac{1}{1-2\rho(n-1)}\varphi$.
\end{lemma}
\begin{proof}
Lemma~\ref{lemvolex1} implies that
\begin{equation*}
\frac{R}{\log\, \operatorname{vol}_{\hat\varphi}(B(p,R))}\notin L^1(+\infty).
\end{equation*}
Indeed, it is enough to note that $\frac{R}{\log\,( A_0+B_0\int^{R}_{r_0}e^{C_0 t^2}dt)}= o\left(\frac{1}{R}\right)$ as $R\rightarrow+\infty$.
So, by Theorem~9 in~\cite{pigola2011remarks} (or from~\cite[Chapter~3]{pigola2005maximum}) the weak maximum principle at infinity is valid for  $\Delta_{\hat\varphi}$ on $(M^n,g).$
\end{proof}

We conclude this section by proving our first main result.
\subsection{Proof of Theorem~\ref{ESRBF}}
\begin{proof}
We know from~\cite{catino2016gradient} that for any gradient $\rho$-Einstein soliton it is true that
\begin{eqnarray*}
\left( \rho(1-n)+\frac{1}{2}\right)\Delta S-\frac{1}{2}g(\nabla S,\nabla \varphi)=\lambda S+\rho S^2-|Ric|^2.
\end{eqnarray*}
Since $|Ric|^2\geq \frac{S^2}{n}$, we get
\begin{equation*}\label{proteo03}
\Delta_{\hat{\varphi}}S\leq \xi_1 S+\xi_2 S^2,
\end{equation*}
where
\begin{eqnarray*}
\hat{\varphi}=\frac{1}{1-2\rho(n-1)}\varphi,\quad \xi_1=\frac{2\lambda}{1-2\rho(n-1)}\quad\hbox{and}\quad \xi_2=\frac{2(\rho-\frac{1}{n})}{1-2\rho(n-1)}.
\end{eqnarray*}
Considering the function $u(x)=\max\{-S(x),0\}$, we have 
\begin{equation*}
\Delta_{\hat\varphi}u\geq \xi_1 u-\xi_2 u^2.
\end{equation*}
Assume the scalar curvature is not necessarily bounded below. Since \eqref{inelel} is valid, by straightforward computation
\begin{equation*}
   \liminf_{R\rightarrow +\infty}{\frac{\log{\operatorname{vol}_{\hat\varphi}(B(p,R))}}{R^2}}<+\infty,
\end{equation*}
so that we can apply Theorem 12 in \cite{pigola2011remarks} from which $u$ is bounded above.
Setting $u^*=\sup_M u<+\infty$, by Lemma~\ref{princwf} we can apply the weak maximum principle at infinity for $\Delta_{\hat\varphi}$ to obtain
\begin{equation*}
0\geq \xi_1 u^*-\xi_2 (u^*)^2,
\end{equation*}
whence
\begin{equation*}
 u\leq\max\left\{0, -\frac{\xi_1}{\xi_2}\right\}=\max\left\{0,\frac{\lambda n}{1-\rho n}\right\}.
\end{equation*}
Therefore
\begin{equation*}
S(x)\geq \min \left\{0,\frac{\lambda n}{1-\rho n}\right\}. 
\end{equation*}
This finishes the proof.
\end{proof}

\begin{remark}
Notice that the hypothesis of scalar curvature bounded below implies that $Ric_{\varphi}$ is also bounded below, and then one can use the $\varphi$-weighed volume growth of geodesic balls by Wei and Wylie~\cite{wei2009comparison} to guarantee the validity of the weak maximum principle at infinity for $\Delta_\varphi$. However, we emphasize that the proof of Theorem~\ref{ESRBF} follows from the $\hat\varphi$-weighed volume growth of geodesic balls together with the weak maximum principle at infinity for $\Delta_{\hat{\varphi}}$, where $\hat{\varphi}=\frac{1}{1-2\rho(n-1)}\varphi$. Owing to this, we need to have some hypothesis about the potential function.
\end{remark}

\section{Gradient Einstein soliton warped products}\label{SecGESWP}
Given two Riemannian manifolds $(B^n,g_B)$ and $(F^m,g_F)$ as well as a positive smooth warping function $f$ on $B^n$, let us consider on the product manifold $B^n\times F^m$ the warped metric
\begin{equation*}
g=\pi_1^*g_B+(f\circ \pi_1)^2\pi^*_2g_F, 
\end{equation*}
where $\pi_1$ and $\pi_2$ are the natural projections on $B^n$ and $F^m$, respectively. Under these conditions, the product manifold is called the \emph{warped product} of $B^n$ and $F^m$. The notation $\tilde\varphi=\varphi\circ\pi_1$ stands for the lift of a smooth function $\varphi$ on $B^n$ to $B^n\times F^m$. If $f$ is a constant, then $(B^n\times F^m,g)$ is a \emph{standard Riemannian product}.

Here we study gradient $\rho$-Einstein solitons that are realized as warped metrics $g$ on $B^n\times F^m$. For simplicity, we say that $B^n\times_f F^m$ is a \emph{gradient $\rho$-Einstein soliton warped product} with a potential function $\eta$, and $g$ is a \emph{gradient $\rho$-Einstein soliton warped metric}. We denote by $Ric_g$ the Ricci tensor of the warped metric $g$, while $Ric_B$ and $Ric_F$ are the Ricci tensors of $g_B$ and $g_F$, respectively. We shall use similar notations for the geometric objects of $g$, $g_B$ and $g_F$.

We recall that a warped product is a geodesically complete manifold for all warping functions if and only if both base and fiber are geodesically complete manifolds, see Bishop and O'Neill~\cite[Lemma~7.2]{bishop1969manifolds}. 

It is known that the potential function of a gradient Ricci soliton warped product with a nonconstant warping function depends only on the base, see Borges and Tenenblat~\cite[Corollary~2.2]{borges2017riccim}. Moreover, the scalar curvature of the fiber is constant, see Feitosa, Freitas Filho and Gomes~\cite{feitosa2017construction}. Notice that our Theorem~\ref{rhoonlybase} is a similar result for a gradient $\rho$-Einstein soliton warped product with a nonconstant warping function. 

We begin with the case of a compact base.  In this case, the idea of the proof is analogous to Borges and Tenenblat's result or, alternatively, see Lemma~4.1 in the paper by the second author in joint work with Marrocos and Ribeiro~\cite{gomes2021note}. 
\begin{proposition}\label{compct-casePF}
Let $B^{n} \times_{f} F^{m}$ be a gradient $\rho$-Einstein soliton warped product with a nonconstant warping function and compact base. Then, the scalar curvature of the fiber is constant. Moreover, the potential function $\eta=\tilde{\varphi}$ and $\lambda+\rho S_g=\tilde{\Lambda}$ for some smooth functions $\varphi$ and $\Lambda$ on $B^n$.
\end{proposition}
\begin{proof}
In order to be able to apply this proof in the more general case of geodesically complete gradient  $\rho$-Einstein soliton warped products, we shall only use the compactness assumption  in the last part of the proof.

Let $g$ be a gradient $\rho$-Einstein soliton warped metric on $B^{n} \times F^{m}$, then 
\begin{equation}\label{eqle}
Ric_g+\nabla^2_g\eta=(\rho S_g+\lambda)g,
\end{equation}
for some constants $\rho\neq0$ and $\lambda$. So, by using Lemmas~7.3 and 7.4 due to Bishop and O'Neill~\cite{bishop1969manifolds}, we obtain
\begin{equation*}
0=\left(\nabla_g^{2} \eta\right)(X, V)=X(V(\eta))-\left(\nabla_{X} V\right)(\eta)=X(V(\eta))-\frac{X(f)}{f} V(\eta),
\end{equation*}
for all $X, V$ horizontal and vertical vector fields, respectively. This nullity implies
\begin{equation*}
X\left(V\left(\frac{\eta}{f}\right)\right)=X\left(\frac{V(\eta)}{f}\right)=\frac{1}{f}\left[X(V(\eta))-\frac{X(f)}{f} V(\eta)\right]=0.
\end{equation*}
Therefore, the function $V\left(\frac{\eta}{f}\right)$ depends only on the fiber $F^m$. Thus, without loss of generality, we can write $\eta=\varphi+f h$, for some smooth functions $\varphi$ on $B^n$ and $h$ on $F^m$, which has been the main idea by Borges and Tenenblat. 

We now observe that, for horizontal vector fields, equation~\eqref{eqle} becomes
\begin{equation}\label{EqPonBase}
Ric_B(X,Y)\!-\!\frac{m}{f}\nabla_B^2f(X,Y)\!+\!\nabla_B^2\varphi(X,Y)\!+\!h\nabla_B^2f(X,Y)=(\lambda\!+\!\rho S_g)g_B(X,Y)
\end{equation}
which follows from Lemma~7.4 in~\cite{bishop1969manifolds} or, alternatively, \cite[Corollary~43]{o1983semi}. This lemma also implies
\begin{equation}\label{eq05rho}
S_g=S_{B}-\frac{2m}{f}\Delta_{B}f+\frac{S_{F}}{f^2}-m(m-1)\frac{|\nabla_{B} f|^2}{f^2}.
\end{equation}
From equations~\eqref{EqPonBase} and \eqref{eq05rho}, we get 
\begin{equation}\label{eqpepa}
V(h)\nabla^2_Bf(X,Y)=\rho V\left(\frac{S_F}{f^2}\right)g_B(X,Y).
\end{equation}
Taking trace on both sides of \eqref{eqpepa} gives
\begin{equation*}
V(h)\Delta_B f=\rho \frac{V(S_F)}{f^2}n.
\end{equation*}
Since $B^n$ is compact, Stokes' theorem implies $V(S_F)=0$, and then $S_F$ is constant. Now, using that $f$ is nonconstant, from \eqref{eqpepa} we obtain $V(h)=0$, i.e., $h$ is constant. We conclude that $\lambda+\rho S_g=\tilde{\Lambda}$ and $\eta=\tilde{\varphi}$ for some functions $\varphi,\Lambda\in C^\infty(B^n)$. 
\end{proof}

The proof of Proposition~\ref{compct-casePF} has the advantage of allowing the generalization to geodesically complete gradient $\rho$-Einstein soliton warped products $B^{n}\times_{f} F^{m}$ with nonconstant warping function $f$. For this, we first observe that $\nabla f$ cannot be parallel, see Proposition~3.6 in~\cite{borges2017riccim}. In particular, if the scalar curvature of the fiber is constant, then we obtain immediately from~\eqref{eqpepa} the same result of Proposition~\ref{compct-casePF}. 

We are now ready to prove the main result of this section.

\subsection{Proof of Theorem~\ref{rhoonlybase}}

\begin{proof}
Here, we shall follow the notations and the results obtained as in Proposition~\ref{compct-casePF}. The proof requires a more careful analysis of equation~\eqref{eqpepa}.

We start by noting that the scalar curvature of the fiber is nonconstant if and only if $h$ is nonconstant, which follows immediately of~\eqref{eqpepa} and the important fact that $\nabla f$ cannot be parallel, see~\cite[Proposition~3.6]{borges2017riccim}. 

Now we claim that if the scalar curvature of the fiber is nonconstant, then the following properties hold:
\begin{itemize}
\item[(i)] $\nabla^2_B f=\frac{a}{f^2}g_B$, for some constant $a\neq 0$;
\item[(ii)] $\rho= \frac{1}{6}$ and $m=2.$
\end{itemize}
Indeed, take a horizontal vector field $X$ such that $g_B(X,\nabla f)=0$. Equation~\eqref{eqpepa} implies 
\begin{equation*}
V(h)\nabla^2_Bf(\nabla_B f, X)=0.
\end{equation*}
Since $h$ is nonconstant, we can find a vertical vector field $V$ such that $V(h)\neq 0$, and then $\nabla^2_Bf(\nabla_B f,X)=0$. Thus, there exists a smooth function $\psi$ on $B^n$ so that $\nabla_B^2 f(\nabla_Bf)=\psi\nabla_B f$. By~\eqref{eqpepa}, one has
\begin{equation}\label{equaqeayu}
\left(V(h)\psi-\rho\frac{V(S_F)}{f^2}\right)|\nabla_B f|^2=0.
\end{equation}
Consider the closed set $A=\{x\in B^n; |\nabla_B f|(x)=0\}$. Note that the set $int(A)$ is empty or, equivalently, $A^C$ is a dense set in $B^n$. Indeed, if there exists $x_0\in int(A)$, then $\nabla_B^2 f(x_0)=0$ that implies $S_F$ is constant (see~\eqref{eqpepa}), which is a contradiction. From~\eqref{equaqeayu} and by continuity, we have
\begin{equation*}
 V(h)\psi-\rho\frac{V(S_F)}{f^2}=0
\end{equation*}
on $B^n\times F^m$. So,
\begin{equation*}
V(h)X(\psi f^2)=0.
\end{equation*}
It follows that $\psi f^2=a$, for some constant $a\neq 0$, and then
\begin{equation}\label{eqpepafibr4}
aV(h)=\rho V(S_F).
\end{equation}
Again from~\eqref{eqpepa}
\begin{equation*}
V(h)\left(\nabla_B^2 f-\frac{a}{f^2}g_B\right)=0
\end{equation*}
and we conclude the proof of item~(i). 

To prove item~(ii), we first obtain from item (i) the identity 
\begin{equation}\label{norconfc}
|\nabla f|^2+\frac{2a}{f}=b,
\end{equation}
for some constant $b.$
    
Using the $\rho$-Einstein soliton equation on $F^m$ (see \cite[Corollary~43]{o1983semi}) and that $\eta=\varphi+fh$, we obtain
\begin{eqnarray}\label{eqpepafibr1}
\nonumber Ric_F-\left(\frac{\Delta_B f}{f}+(m-1)\frac{|\nabla_B f|^2}{f^2}\right) f^2 g_F +f \nabla_F^2 h + f\Big(\nabla_Bf(\varphi)+h|\nabla_B f|^2\Big)g_F\\
=(\lambda+\rho S_g)f^2 g_F.
\end{eqnarray}
Keep in mind that the scalar curvature $S_g$ is given by~\eqref{eq05rho}, and recall the second contracted Bianchi identity $\operatorname{div}_FRic_F= \frac{\nabla_F S_F}{2}$. Thus, we use \eqref{eqpepafibr1} to obtain
\begin{equation}\label{eqpepafibr2}
\frac{\nabla_F S_F}{2}+f \operatorname{div}_F\nabla_F^2 h +f|\nabla_B f|^2\nabla_F h=\rho \nabla_F S_F.
\end{equation}
Substituting \eqref{norconfc} into~\eqref{eqpepafibr2} yields
\begin{equation*}
 \left(\frac{1}{2}-\rho\right)\nabla_F S_F+f \operatorname{div}_F\nabla_F^2h +f\left(b-\frac{2a}{f}\right)\nabla_F h=0.
\end{equation*}
On the other hand \eqref{eqpepafibr4} implies $\nabla_FS_F=\frac{a}{\rho}\nabla_Fh$, and then
\begin{equation*}
a\left(\frac{1}{2\rho} -3
\right) \nabla_F h+f \Big(\operatorname{div}_F\nabla_F^2 h +b\nabla_F h\Big)=0.
\end{equation*}
As $f$ is nonconstant and it depends only on $B^n$, we get
\begin{equation*}
 a\left(\frac{1}{2\rho}-3
\right) \nabla_F h=0.
\end{equation*}
 Hence, $\rho=-1/6
$. Let us now prove that $m=2$. We take traces in \eqref{eqpepafibr1} and we use the value of $\rho$ to get
\begin{eqnarray*}
S_F-\left(\frac{\Delta_B f}{f}+(m-1)\frac{|\nabla_B f|^2}{f^2}\right)f^2 m+f \Delta_F h +f\Big(\nabla_Bf(\varphi)+h|\nabla_B f|^2\Big)m\\
=\left(\lambda +\frac{S_g}{6}\right)f^2 m.
\end{eqnarray*}
Whence
\begin{equation*}
\left(1 -\frac{m}{6}\right)\nabla_F S_F+f \nabla_F\Delta_F h +f|\nabla_B f|^2 m \nabla_F h=0.
\end{equation*}
Substituting $\nabla_FS_F= 6a\nabla_F h$ and \eqref{norconfc} into the previous equation gives
\begin{equation*}
-3a\left(-2+m\right)\nabla_F h+f\Big(\nabla_F\Delta_F h +bm\nabla_Fh\Big)=0.
\end{equation*}
Hence, $m=2$, and this completes the proof of item  (ii), which in turn shows items~\eqref{rhoonlybase-1} and \eqref{rhoonlybase-1-1} of the theorem.

For item~\eqref{rhoonlybase-2} of the theorem, it is enough to prove that $f$ does not satisfy items (i) and (ii) when it is bounded. 

We start by observing that $\nabla^2_B f=\frac{a}{f^2}g_B$, for some constant $a\neq 0$, and we can assume $a>0$. Now, we take a unit vector $v\in T_{p} B$ such that $d f_{p} v> 0$, $p \in B^n$. Let $\gamma$ be the unit speed geodesic on $B^n$ such that $\gamma(0)=p$ and $\dot\gamma(0)=v$. Define $\beta(t)=f(\gamma(t))$. Then, from (i) one has
\begin{equation*}
\ddot\beta(t)=\nabla_B^2 f(\dot\gamma(t),\dot\gamma(t))=\frac{a}{\beta^2(t)}>0.
\end{equation*}
Thus, $\dot\beta(t)>0$ for all $t\geq 0$ and $\beta$ is an increasing function on $[0,\infty)$ and bounded, since $f$ is bounded. Hence $\beta$ has a horizontal 
asymptote. In particular,
\begin{equation*}
 0=\lim_{t\rightarrow+\infty}\dot\beta(t)=\lim_{t\rightarrow+\infty}\left(\dot\beta(0)+\int_0^t\frac{a}{\beta^2(s)}ds\right)>\dot\beta(0)=df_p(v)>0,
\end{equation*}
which is a contradiction. This completes the proof of the theorem.
\end{proof}

\subsection{Necessary and sufficient conditions to construct a gradient Einstein soliton warped product.}

We start with some known results on gradient Ricci almost solitons. A Riemannian manifold $(M^n,g)$ is a gradient Ricci almost soliton if there exist two smooth functions $\eta$ and $\tau$ on $M^n$ satisfying
\begin{equation*}
Ric+\nabla^2\eta=\tau g.
\end{equation*}
From Theorem~\ref{rhoonlybase}, it is clear that a gradient $\rho$-Einstein soliton warped product is a very special case of gradient Ricci almost soliton warped product. We recall that Ricci almost solitons were first defined in 2011 by Pigola, Rigoli, Rimoldi, and Setti in an article bearing the same name \cite{pigola2011ricci}.

Let us consider a gradient Ricci almost soliton warped product $B^n\times_f F^m$, with $\eta=\tilde\varphi$ and $\tau=\tilde\Lambda$, for some smooth functions $\varphi$ and $\Lambda$ on $B^n$. The second author in joint work with Feitosa, Freitas Filho, and Pina~\cite{feitosa2019gradient} established the necessary and sufficient conditions for constructing such a soliton. More precisely, the following two equations 
\begin{equation}\label{eq01rho}
R i c_{B}+\nabla_{B}^{2} \varphi-\frac{m}{f} \nabla_{B}^{2} f=\Lambda g_{B},
\end{equation}
\begin{equation}\label{eq03rho}
-2 \Lambda d \varphi+d\left((2-m-n)\Lambda+|\nabla_{B} \varphi|^{2}-\Delta_{B} \varphi-\frac{m}{f} \nabla_{B} \varphi(f)\right)=0,
\end{equation}
must be satisfied on $B^n$. Moreover, $Ric_F=\mu g_F$, with constant $\mu$ given by
\begin{equation}\label{eq02rho}
\mu=\Lambda f^{2}+f \Delta_{B} f+(m-1)|\nabla_{B} f|^{2}-f \nabla_{B} \varphi(f).
\end{equation}
They proved that equation~\eqref{eq02rho} is a consequence of~\eqref{eq01rho} and~\eqref{eq03rho}. Now, we look for a necessary condition for $\Lambda$ to be a linear function of the scalar curvature $S_g$ of $B^n\times_f F^m$, which satisfies equation~\eqref{eq05rho} on $B^n$. Taking the trace of~\eqref{eq01rho}, we get
\begin{equation}\label{sufrho02}
S_B=n\Lambda-\Delta_B \varphi+\frac{m}{f}\Delta_B f.
\end{equation}
From equations~\eqref{eq05rho}, \eqref{eq02rho} and~\eqref{sufrho02}, we obtain
\begin{equation}\label{sufrho03}
S_g=(n+1)\Lambda-\left(\Delta_B\varphi+\frac{\nabla_B \varphi(f)}{f}\right)-(m-1)\left(\frac{\Delta_B f}{f}+(m-1)\frac{|\nabla_B f|^2}{f^2}\right).
\end{equation}
Since we are looking for conditions for $\Lambda$ to be a linear function of $S_g$ up to a constant, the condition that arises naturally from~\eqref{sufrho03} is
\begin{equation}\label{eq04rho}
\alpha\nabla_B\Lambda=\nabla_B\left(\Delta_B\varphi+\frac{\nabla_B \varphi(f)}{f}+(m-1)\left(\frac{\Delta_B f}{f}+(m-1)\frac{|\nabla_B f|^2}{f^2}\right)\right),
\end{equation}
for some constant $\alpha\neq n+1$. 

The reason why we take $\alpha\neq n+1$  comes from the fact that when $\alpha=n+1$, then the scalar curvature $S_g$ is constant. Indeed, by~\eqref{eq04rho} we get
\begin{equation*}
(n+1)\Lambda=\Delta_B\varphi+\frac{\nabla_B \varphi(f)}{f}+(m-1)\left(\frac{\Delta_B f}{f}+(m-1)\frac{|\nabla_B f|^2}{f^2}\right)+c,
\end{equation*}
for some constant $c$. Substituting this last equation into~\eqref{sufrho03}, we obtain $S_g=c$. In particular, this reduces the problem of construction of the gradient $\rho$-Einstein soliton warped products to the case of gradient Ricci solitons, and then equation~\eqref{eq04rho} is unnecessary, see \cite[Theorem~3]{feitosa2017construction}.

First, we prove the necessary conditions for constructing a gradient $\rho$-Einstein soliton warped product. 

\begin{proposition}\label{condsufrho}
Let $B^{n}\times_{f}F^{m}$ be a gradient $\rho$-Einstein soliton warped product with potential function $\tilde\varphi$ and soliton function $\tilde\Lambda=\lambda+\rho\tilde S_g$. Then equations~\eqref{eq01rho}, \eqref{eq03rho} and \eqref{eq04rho} are satisfied on $(B^n,g_B)$, and $Ric_F=\mu g_F$, where $\mu$ is the constant given by equation~\eqref{eq02rho}.
\end{proposition}
\begin{proof}
Equation~\eqref{eq01rho} and the fact that $Ric_F=\mu g_F$ with constant $\mu$ given by~\eqref{eq02rho}, follows directly by Proposition~2 in~\cite{feitosa2019gradient}, while \eqref{eq03rho} is equation~(1.3) in~\cite{feitosa2019gradient} on $B^n$. By taking the covariant derivative of equation~\eqref{sufrho03}, we  conclude that~\eqref{eq04rho} is satisfied for $\alpha=n+1-\frac{1}{\rho}$.
\end{proof}

Second, we prove the sufficient conditions for constructing a gradient $\rho$-Einstein soliton warped product. 
\begin{proposition}\label{constrhoe}
Let $\left(B^{n}, g_{B}\right)$ be a Riemannian manifold with three smooth functions $f>0, \Lambda$, and $\varphi$ satisfying equations~\eqref{eq01rho}, ~\eqref{eq03rho} and~\eqref{eq04rho}. Take the constant $\mu$ satisfying equation~\eqref{eq02rho} and a Riemannian manifold $\left(F^{m}, g_{F}\right)$ with Ricci tensor $Ric_F=\mu g_{F}$. Then $B^{n} \times_{f} F^{m}$ is a gradient $\rho$-Einstein soliton warped product with potential function $\tilde\varphi$, soliton function $\tilde\Lambda$ and $\rho=\frac{1}{n+1-\alpha}$.
\end{proposition}
\begin{proof}
Since $f$, $\varphi$ and $\Lambda$ satisfy~\eqref{eq01rho} and~\eqref{eq03rho} on $B^n$, and $Ric_F=\mu g_F$ with $\mu$ given by~\eqref{eq02rho}, Theorem~1 in~\cite{feitosa2019gradient} guarantees that $B^{n}\times_{f} F^{m}$ is a gradient Ricci almost soliton warped product with potential function $\tilde\varphi$ and soliton function $\tilde\Lambda$. In particular, equation~\eqref{sufrho03} is valid, and then by~\eqref{eq04rho} we get
\begin{equation*}
S_g=(n+1-\alpha)\Lambda+c,
\end{equation*}
for some constant $c$. Now, it is enough to define $\rho=\frac{1}{n+1-\alpha}$ and $\lambda=\Lambda-\frac{1}{n+1-\alpha}S_g$ to obtain a gradient $\rho$-Einstein soliton warped product.
\end{proof}

Next, we give an immediate application of Proposition~\ref{condsufrho}.

\begin{corollary}\label{teofitrwp} 
Let $B^n\times_f F^m$ be a gradient $\rho$-Einstein soliton warped product which is either shrinking or steady, with potential function $\tilde\varphi$, nonconstant warping function and nonnegative scalar curvature. If $f$ reaches a minimum and $\rho>0$, then the fiber has positive scalar curvature. In particular, $F^m$ is a compact manifold.
\end{corollary}
\begin{proof}
By Proposition~\ref{condsufrho}, $Ric_F=\mu g_F$, where $\mu$ is the constant given by \eqref{eq02rho} on $B^n$, i.e.,
\begin{equation*}
\frac{\mu}{f}-\Lambda f=\Delta_B f+\frac{m-1}{f}|\nabla_B f|^2-\nabla_B\varphi(f).
\end{equation*}
By assumption $\rho>0$, $S_g\geq0$ and $\lambda> 0$ or $\lambda=0$ that imply $\Lambda=\lambda+\rho S_g\geq 0$. If $\mu\leq 0$ we have
\begin{equation*}
0\geq\Delta_B f+\frac{m-1}{f}|\nabla_B f|^2-\nabla_B\varphi(f).
\end{equation*}
If $f$ reaches a minimum, then it must be a constant by the strong maximum principle, which is a contradiction. So, $\mu>0$, and then $F^m$ is compact, since it is an Einstein manifold.
\end{proof}

\section{Examples}\label{Sec-Examples}
In this section, we present some examples of gradient $\rho$-Einstein soliton warped metrics. For this purpose, we use Proposition~\ref{constrhoe} and we consider bases that are conformal to Euclidean spaces $\mathbb{R}^n$ in the same way as in \cite[Proposition~4]{feitosa2019gradient}.

\begin{example}\label{traceless1}
Let us consider the hyperbolic space $\mathbb{H}^n$, namely, the open half space $x_n >0$ with the metric $\coth^2(x_n)\delta_{ij}$. The functions
\begin{align*}
f&=\coth(x_n),\\
\varphi&=\frac{2}{3}(-2+m+n)\log\big(\cosh(x_n)\big) \quad \hbox{and}\\
\Lambda&=-\frac{1}{3 \cosh^4(x_n)}\big(2(-2+m+n)+(1+m+n)\cosh(2x_n)\big)
\end{align*}
satisfy equations~\eqref{eq01rho} and ~\eqref{eq03rho} on $\mathbb{H}^n$. Besides, the constant $\mu$ given by equation~\eqref{eq02rho} is null. Theorem~1 in~\cite{feitosa2019gradient} guarantees that, for any geodesically complete Ricci flat Riemannian manifold $F^m$, the warped metric $g$ with warping function $f$, potential function $\tilde\varphi$ and soliton function $\tilde\Lambda$ is a geodesically complete gradient Ricci almost soliton on $\mathbb{H}^n\times F^m$. The proof is a straightforward computation. Moreover, from~\eqref{eq05rho} its scalar curvature is
\begin{equation*}
S_g=-\frac{(-1+m+n)(-2+m+n+2\cosh(2x_n))}{\cosh^4(x_n)}.
\end{equation*}
The special case of $\rho$-Einstein solitons occurs when $n=1$ and $m=2$, since equation~\eqref{eq04rho} is satisfied for $\alpha=-1$ on $\mathbb{H}^1$, so that $\mathbb{H}^1\times_f F^2$ is a nontrivial complete gradient steady traceless Ricci soliton warped product. Indeed, the completeness of $\nabla\varphi$ is a straightforward computation. Using the divergent curve criterion we prove that $\mathbb{H}^1 \times_{f} F^2$ is geodesically complete. For this, it is enough to prove that the metric $\coth^2(t)dt^2$ is geodesically complete since the fiber is chosen to be geodesically complete. Let $\gamma:[t_0,\infty)\rightarrow \mathbb{H}^1$ be a divergent curve. First, we assume $\gamma(t)\rightarrow+\infty$ as $t\rightarrow+\infty$, and we suppose without loss of generality that $\gamma'(t)>0$,  $\forall t\geq t_0$. Then
\begin{eqnarray*}
\int^t_{t_0}\coth(t)|\gamma'(t)|dt &=& \int^t_{t_0}\coth(t)\gamma'(t)dt\\
&=& \gamma(t)\coth(t)-\gamma(t_0)\coth(t_0)+\int^t_{t_0}\frac{\gamma(t)}{\sinh^2(t)}dt\\
&\geq& \gamma(t)\coth(t)-\gamma(t_0)\coth(t_0).
\end{eqnarray*}
Hence
\begin{equation*}
\lim_{t\rightarrow+\infty}\int_{t_0}^t \coth(t)|\gamma'(t)|dt=+\infty.
\end{equation*}
The case $\gamma(t)\rightarrow 0$ as $t\to \infty$ is similar, and then we obtain the required completeness of $\mathbb{H}^1\times_f F^2$.
\end{example}

\begin{example}\label{traceless2} Consider the metric $\cosh^2(x_n)\delta_{ij}$ on $\mathbb{R}^n$. The functions
\begin{align*}
f&=\cosh(x_n),\\
\varphi&=\frac{1}{12}(-2+m+n)\left(8\log\left(\cosh(x_n)\right)+\cosh(2x_n)\right)\quad\hbox{and}\\
\Lambda&=-\frac{3-(-2+m+n)\sinh^4(x_n)}{3\cosh^4(x_n)}
\end{align*}
satisfy equations~\eqref{eq01rho} and~\eqref{eq03rho} on $\mathbb{R}^n$. Besides, the constant $\mu$ given by equation~\eqref{eq02rho} is null. Again, Theorem~1 in~\cite{feitosa2019gradient} guarantees that, for any geodesically complete Ricci flat Riemannian manifold $F^m$, the warped metric $g$ with warping function $f$, potential function $\tilde\varphi$ and soliton function $\tilde\Lambda$ is a complete gradient Ricci almost soliton on $\mathbb{R}^n\times F^m$. The proof is a straightforward computation. Moreover, from~\eqref{eq05rho} its scalar curvature is
\begin{equation*}
S_g=-\frac{(-1+m+n)(6-m-n+(-2+m+n)\cosh(2x_n))}{2\cosh^4(x_n)}.
\end{equation*}
Again, the special case of $\rho$-Einstein solitons occurs when $n=1$ and $m=2$, since equation~\eqref{eq04rho} is satisfied for $\alpha=-1$ on $(\mathbb{R},\cosh^2(t)dt^2)$, so that $\mathbb{R}\times_f F^2$ is a nontrivial complete gradient shrinking traceless Ricci soliton warped product with $\lambda=1/3$. We can prove its completeness by an analogous argument as in Example~\ref{traceless1}.
\end{example}

\begin{example}\label{exmplo2res}
Let $\mathbb{R}_+^{n}$, $n \geq 3$, be a Euclidean half space with the canonical coordinates $x=\left(x_{1}, \ldots, x_{n}\right)$ such that $x_n>0$ and the metric $\frac{1}{x_n^4}\delta_{i j}$. The functions
\begin{equation*}
f=\frac{1}{x_n}, \quad \varphi=\frac{2(2-m-n)}{3} \ln{x_n} \quad \mbox{and} \quad \Lambda=\frac{2(5-m-4n)x_n^2}{3}
\end{equation*}
satisfy equations~\eqref{eq01rho},~\eqref{eq03rho} and \eqref{eq04rho} on $\mathbb{R}_+^{n}$ for 
\begin{equation*}
\alpha=-\frac{17+3m^2-18n+4n^2+2m(-8+5n)}{2(4n+m-5)}.    
\end{equation*}
The constant $\mu$ given by equation~\eqref{eq02rho} is $\mu=\frac{5-m-4n}{3}$. Proposition~\eqref{constrhoe} guarantees that, for any Einstein manifold $F^m$, with $Ric_F=\mu g_F$, the warped metric $g$ with warping function $f$, potential function $\tilde\varphi$ and soliton function $\tilde\Lambda$ is a gradient $\rho$-Einstein steady soliton warped metric on $\mathbb{R}^n_+\times F^m$, with $\rho=\frac{1}{n+1-\alpha}$. Moreover, from~\eqref{eq05rho} its scalar curvature is 
\begin{eqnarray*}
S_g&=&-\frac{1}{3}\left(7+3m^2-20n+12n^2+2m(-7+6n)\right)x_n^2.
\end{eqnarray*}
In this case, the warped metric $g$ is incomplete, since the curve $\gamma:[1,\infty)\rightarrow(0,\dots,0,t)$ on $\mathbb{R}^n_+$ satisfies
\begin{eqnarray*}
\lim_{r\rightarrow+\infty}\int_1^r\frac{1}{t^2}dt=1.
\end{eqnarray*}
\end{example}

As we already observed in the introduction, the next example shows that the assumption of completeness is essential for the results by Catino and Mazzieri~\cite[Theorem 1.5]{catino2016gradient} and by Borges~\cite[Theorem 1.1]{valterborges2022}

\begin{example}\label{exemimcp}
Let $\mathbb{R}^{n}$ be a Euclidean space with coordinates $x=\left(x_{1}, \ldots, x_{n}\right)$ and metric $e^{2 \xi} \delta_{i j}$, where $\xi=\sum_{i=1}^{n} \alpha_{i} x_{i}, \alpha_{i} \in \mathbb{R}$ with $\sum_{i}\alpha^2_i=1$ and $n \geq 3$. The functions
\begin{equation*}
f=e^{\xi}, \quad \varphi=\frac{c}{2} e^{2 \xi}-\frac{(2-m-n)}{2} \xi \quad \text {and} \quad \Lambda=c+\frac{(2-m-n)}{2} e^{-2 \xi}
\end{equation*}
satisfy equations~\eqref{eq01rho},~\eqref{eq03rho} and \eqref{eq04rho} on $\mathbb{R}^{n}$ for $\alpha=-2m-n+3$ and any constant $c$. The constant $\mu$ given by equation~\eqref{eq02rho} is null. Proposition~\eqref{constrhoe} guarantees that, for any Ricci flat Riemannian manifold $F^m$, the warped metric $g$ with warping function $f$, potential function $\tilde\varphi$ and soliton function $\tilde\Lambda$ is a gradient Schouten soliton warped metric on $\mathbb{R}^n\times F^m$ which can be either steady, shrinking or expanding, since $\lambda=c$. Moreover, from~\eqref{eq05rho} its scalar curvature is 
\begin{eqnarray*}
S_g=-(m+n-2)(m+n-1)e^{-2\xi}.
\end{eqnarray*}
To show that the warped metric $g$ is incomplete, we can use a similar argument as in Example~\ref{exmplo2res}.
\end{example}

\section{Asymptotic behavior of the potential function}\label{Abpf}

In this section, we prove Theorem~\ref{teocresc2}. We start by proving an interesting result for a class of noncompact geodesically complete Riemannian manifold $(B^n,g_B)$. For this, we consider the following modification of the $m$-Bakry-Emery Ricci tensor, which stems from the construction of both gradient Ricci solitons (see~\cite{feitosa2017construction,gomes2021note}) and gradient $\rho$-Einstein solitons (see Proposition~\ref{condsufrho}) that are realized as warped metrics:
\begin{equation*}
Ric^m_{h, \varphi}:= Ric + \nabla^2\varphi +\nabla^2h-\frac{1}{m}dh\otimes dh=Ric+\nabla^2\varphi-\frac{m}{f}\nabla^2f,
\end{equation*}
where $m$ is a positive integer number, $\varphi$ and $h$ are smooth functions on $B^n$ and $f=e^{-h/m}$.

The next proposition is motivated by the results proved in Section~\ref{SecGESWP}.

\begin{proposition}\label{riccihes}
Let $(B^n,g_B)$ be a noncompact geodesically complete Riemannian manifold that admits two smooth functions $h$ and $\varphi$ such that
\begin{equation}\label{eqriches01}
Ric_{h,\varphi}^m\geq \lambda,
\end{equation}
for some positive constant $\lambda$. Fix a point $p\in B^n$, and take a minimizing unit speed geodesic $\gamma:[0,\infty)\rightarrow B^n$ emanating from $p$.
Then,
\begin{equation*}
    \limsup_{t\rightarrow+\infty}\frac{1}{t}\frac{d\varphi}{dt}\geq\lambda,
\end{equation*}
uniformly along $\gamma$, i.e., it does not depend on the initial direction of the geodesic. 
\end{proposition}
\begin{proof}
The main idea is based on the proof of Proposition~2.1 in~\cite{cao2010complete}, in our case has been necessary to control the extra term $\nabla^2h-\frac{1}{m}dh\otimes dh$. 

Take $r(x)$ the distance function from a point $p\in B^n$, and consider $\gamma:[0,r(x)]\rightarrow B^n$ a minimizing unit speed geodesic, with $\gamma(0)=p$. By the second variation of the distance function, we have
\begin{equation}\label{eqgreo0}
 0\leq\int_{0}^{r}\Big(\dot{\phi}^{2}(t)(n-1)-\phi^2(t) Ric(\gamma',\gamma')\Big) dt,
\end{equation}
for every smooth piecewise function $\phi(t)$ on $\left[0, r\right]$ with $\phi(0)=\phi(r)=0$ and $r>1$. We choose $\phi(t)$ given by
\begin{equation}\label{pessof}
\phi(t)= \begin{cases}t, & t \in[0,1], \\ 1, & t \in\left[1, r-1\right], \\ r-t, & t \in\left[r-1, r\right] .\end{cases}
\end{equation}
Using~\eqref{eqriches01}, \eqref{eqgreo0}, \eqref{pessof}, and integration by parts, we get
\begin{align}\label{limfroeipo1}
\nonumber 0&\leq \int_0^r\Big(\dot\phi^2(n-1)-\phi^2 Ric(\dot\gamma,\dot\gamma)\Big)dt\\
\nonumber&\leq\int_0^r \left(\dot\phi^2(n-1)+\phi^2\Big( -\lambda+\frac{d^2h}{dt^2}-\frac{1}{m}\left(\frac{dh}{dt}\right)^2+\frac{d^2 \varphi}{dt^2}\Big)\right)dt\\
\nonumber&=\int_0^r \left(\dot\phi^2(n-1)-2\phi\dot\phi\frac{dh}{dt}+\phi^2\Big( -\lambda-\frac{1}{m}\left(\frac{dh}{dt}\right)^2+\frac{d^2 \varphi}{dt^2}\Big)\right)dt\\
&\leq\int_0^r \left(\dot\phi^2(n-1)+2\left|\phi\dot\phi\frac{dh}{dt}\right|+\phi^2\Big( -\lambda-\frac{1}{m}\left(\frac{dh}{dt}\right)^2+\frac{d^2 \varphi}{dt^2}\Big)\right)dt.
\end{align}
Applying the elementary inequality $2 \sqrt{a b} \leq a+b$,  for $a=m\dot\phi^2$ and $b=\frac{1}{m}\phi^2\left(\frac{dh}{dt}\right)^2$ we have
\begin{equation*}
2\int_{0}^{r} \left|\phi\dot{\phi}\frac{dh}{dt}\right| d t \leq \int_{0}^{r}\left(m \dot{\phi}^{2}+\frac{1}{m} \phi^{2}\left(\frac{dh}{dt}\right)^2\right) dt .
\end{equation*}
Combining this previous inequality and~\eqref{limfroeipo1} we get
\begin{align}\label{ines01}
\nonumber 0&\leq\int_0^r \left(\dot\phi^2(m+n-1)+\phi^2\Big( -\lambda+\frac{d^2 \varphi}{dt^2}\Big)\right)dt\\
&=2(m+n-1)- \lambda\left(r-\frac{1}{3}\right)+\int_0^r \phi^2\frac{d^2 \varphi}{dt^2}dt.
\end{align}
Analyzing
\begin{align}\label{ines02}
\nonumber\int_0^r \phi^2\frac{d^2 \varphi}{dt^2}dt&=\int_0^1 \Big(-2\phi\dot\phi\frac{d\varphi}{dt}\Big)dt+\int_{r-1}^r \Big(-2\phi\dot\phi\frac{d\varphi}{dt}\Big)dt\\
\nonumber&\leq\sup_{x\in B(p,1)}|\nabla \varphi|+2\int_{r-1}^r (r-t)\frac{d\varphi}{dt}dt\\
&\leq \sup_{x\in B(p,1)}|\nabla \varphi|+\sup_{t\in[r-1,r]}\frac{d\varphi}{dt}.
\end{align}
On the other hand,
\begin{eqnarray*}
\frac{1}{r}\sup_{t\in[r-1,r]}\frac{d\varphi}{dt}\leq \sup_{t\in[r-1,r]}\frac{1}{t}\frac{d\varphi}{dt}\leq\sup_{t\geq r-1}\frac{1}{t}\frac{d\varphi}{dt}.
\end{eqnarray*}
Combining this latter inequality, \eqref{ines01} and~\eqref{ines02}, we obtain
\begin{equation*}
    \lambda\leq\limsup_{t\rightarrow+\infty}\frac{1}{t}\frac{d\varphi}{dt}.
\end{equation*}
This completes the proof of the proposition.
\end{proof}

Now we are in a position to prove an asymptotic behavior of the potential function on a noncompact gradient shrinking $\rho$-Einstein soliton of nonnegative scalar curvature and  $\rho>0$.

\subsection{Proof of Theorem~\ref{teocresc2}}
\begin{proof} By assumption $S\geq 0$ and $\rho>0$, then
\begin{equation*}
Ric+\nabla^2\varphi\geq \lambda g.
\end{equation*}
Since $\lambda>0$ by Proposition~\ref{riccihes} we obtain~\eqref{crepote}, and applying Corollary~5.1 in~\cite{wei2009comparison}, we conclude that $M^n$ has finite fundamental group and that $\operatorname{vol}_\varphi(M)$ is finite. In particular, 
\begin{equation*}
\frac{R}{\log\, \operatorname{vol}_{\varphi}(B(p,R))}\notin L^1(+\infty).
\end{equation*}
Hence, from~\cite[Chapter~3]{pigola2005maximum} (or alternatively \cite[Theorem~9]{pigola2011remarks}) the weak maximum principle at infinity for $\Delta_\varphi$ holds on $(M^n,g)$.
\end{proof}

\subsection{Proof of Corollary~\ref{Comp-RES}}

\begin{proof}
The proof goes by contradiction. Suppose $M^n$ is noncompact. Theorem~\ref{teocresc2} guarantees that for any minimizing unit speed geodesic $\gamma:[0,\infty)\rightarrow M^n$ emanating from $p\in M^n$, we get
\begin{eqnarray*}
\limsup_{t\rightarrow+\infty}\frac{|\nabla \varphi(\gamma(t))|}{t}\geq \limsup_{t\rightarrow+\infty}\frac{1}{t}\frac{d\varphi}{dt}\geq \lambda.
\end{eqnarray*}
On the other hand, by assumption
\begin{equation*}
\limsup_{t\rightarrow+\infty}\frac{|\nabla \varphi(\gamma(t))|}{t}\leq \lambda-\varepsilon.
\end{equation*}
This leads us to a contradiction.
\end{proof}

\begin{corollary}\label{cricomphowp}
Let $B^n\times_{f}F^{m}$ be a geodesically complete gradient shrinking $\rho$-Einstein soliton warped product with nonnegative scalar curvature, potential function $\tilde\varphi$ and $\rho>0$. Suppose there exists a point $p\in B^n$ such that for all $r(x)\geq r_0$, $|\nabla\varphi(x)|\leq (\lambda-\varepsilon)r(x)+c$, for some positive constants $c$ and $\varepsilon$, with $\varepsilon\leq\lambda$, where $r(x)$ is the distance function from $p$ and $r_0$ is a sufficiently large positive integer. Then $B^n\times_f F^m$ is a compact manifold.
\end{corollary}
\begin{proof}
Since $S_g\geq0$, we get
\begin{eqnarray*}
Ric_B+\nabla^2 \varphi-\frac{m}{f}\nabla^2 f=(\lambda+\rho S_g)g_B\geq\lambda g_B.
\end{eqnarray*}
First, we conclude that $B^n$ is compact by contradiction from Proposition~\ref{riccihes} and an analogous argument to the proof of Corollary~\ref{Comp-RES}. Now we apply the weak maximum principle to $f$ in~\eqref{eq02rho} to conclude that $\mu>0$. Thus, by Bonnet-Myers theorem $F^m$ is compact. Therefore $B^n\times_f F^m$ is compact.
\end{proof}

In the more general setting of gradient Ricci almost solitons, we have:
\begin{corollary}\label{Comp-RES-Base}
Let $B^{n} \times_{f}F^{m}$ be a geodesically complete gradient Ricci almost soliton warped product with potential function $\tilde\varphi$ and soliton function $\tilde\lambda$ satisfying $\lambda^*=\inf_{x\in B}\lambda(x)>0$. Suppose there exists a point $p \in B^n$ such that for all $r(x)\geq r_0$, $|\nabla\varphi(x)|\leq (\lambda^{*}-\varepsilon)r(x)+c$, for some positive constants $c$ and $\varepsilon$, with $\varepsilon\leq\lambda^*$, where $r(x)$ is the distance function from $p$  and $r_0$ is a sufficiently large positive integer. Then $B^n\times_f F^m$ is a compact manifold.
\end{corollary}
\begin{proof}
Applying Proposition 2 in~\cite{feitosa2019gradient}, we have the proof analogous to the demonstration of Corollary~\ref{cricomphowp}.
\end{proof}

\begin{remark}\label{remaramb}
Note that Corollary~\ref{cricomphowp} has the following alternative proof. First, we observe that, in the same way as in the Ambrose compactness criteria~\cite{ambrose1957theorem}, we affirm that if for every geodesic ray $\gamma:[1,\infty)\rightarrow B^n$, we have
\begin{equation*}
\lim_{t\rightarrow +\infty}\int_1^t Ric^m_h(\gamma',\gamma')d\xi=+\infty,
\end{equation*}
then $B^n$ is compact, see Cavalcante, Oliveira and Santos~\cite{cavalcante2015compactness}. Now we can prove that, under the hypotheses of Corollary~\ref{cricomphowp}, this criteria is satisfied on $B^n$, and then $B^n$ is compact, while the compactness of the fiber $F^m$ follows again by Bonnet-Myers theorem, so $B^n\times_{f}F^{m}$ is compact.
\end{remark}

\section*{Acknowledgements}
The authors would like to express their sincere thanks to the anonymous referee for careful reading and useful comments which improved this paper. They are also grateful to Professor Dragomir Mitkov Tsonev for useful comments, discussions and constant encouragement. José N. V. Gomes has been partially supported by Conselho Nacional de Desenvolvimento Científico e Tecnológico (CNPq), of the Ministry of Science, Technology and Innovation of Brazil (Grant 310458/2021-8), and by Fundação de Amparo à Pesquisa do Estado de São Paulo - FAPESP (Grant 2023/11126-7).

\end{document}